\documentclass[a4paper,12pt]{amsart}
\usepackage[margin=2.8cm]{geometry}
\usepackage{fullpage}
\usepackage[applemac]{inputenc}
\usepackage{eurosym}
\newtheorem{theorem}{Theorem}[section]
\newtheorem{rem} [theorem] {Remark}

\newtheorem{cor} [theorem]{Corollary}

\newcommand{\ovprt}{\overline{\partial}}
\newcommand{\ovli}{\overline}
\newcommand{\dquer}{\overline\partial}

\numberwithin{equation}{section}

\title{ Sobolev inequalities and the  $\ovprt$-Neumann operator}

\author{ Friedrich Haslinger}

\thanks{Partially supported by the FWF-grant  P23664.}

 \address{ F. Haslinger: Fakult\"at  f\"ur Mathematik, Universit\"at Wien,
Oskar-Morgenstern-Platz 1, A-1090 Wien, Austria}
\email{ friedrich.haslinger@univie.ac.at}

\keywords{$\ovprt $-Neumann problem, Sobolev inequalities, compactness}
\subjclass[2010] {Primary 32W05; Secondary  30H20, 35P10}

\begin{document}

\maketitle

\begin{abstract} ~\ We study a complex valued version of the Sobolev inequalities and its relationship to compactness of the $\ovprt$-Neumann operator.
For this purpose we use an abstract characterization of compactness derived from a general description of precompact subsets in $L^2$-spaces. Finally we remark that the $\ovprt$-Neumann operator can be continuously extended provided a subelliptic estimate holds.
\end{abstract}

\section{Introduction.}~\\
Let  $\Omega $ be a bounded open set in $\mathbb{R}^n,$ and $k$  a nonnegative integer. We denote by
$W^k(\Omega)$
 the Sobolev space
$$W^k(\Omega) = \{ f\in L^2(\Omega ) \, : \, \partial^\alpha f \in L^2(\Omega ) ,  \, |\alpha |\le k \},$$
where the derivatives are taken in the sense of distributions and endow the space with the norm
\begin{equation*}
\|f\|_{k,\Omega} = \left ( \sum_{|\alpha |\le k} \int_\Omega |\partial^\alpha f |^2 \,d\lambda \right )^{1/2},
\end{equation*}
where $\alpha =(\alpha_1, \dots ,\alpha_n)$ is a multiindex , $|\alpha |=\sum_{j=1}^n \alpha_j$ and 
$$\partial^\alpha f =\frac{\partial^{|\alpha |}f}{\partial x_1^{\alpha_1}\dots \partial x_n^{\alpha_n}}.$$
$W^k(\Omega)$ is a Hilbert space.
 If $\Omega \subset \mathbb{R}^n \, , \, n\ge 2,$  is a bounded domain with a $\mathcal{C}^1$ boundary, the Rellich-Kondrachov lemma says that for $n>2$ one has
$$W^1(\Omega ) \subset L^r(\Omega) \ , \ r\in [1, 2n/(n-2))$$
and that the imbedding is also compact; for $n=2$ one can take $r\in [1,\infty)$ (see for instance \cite{Br}), in particular, there exists a constant $C_r$ such that 
\begin{equation}\label{eq: sob1}
\|f\|_r \le C_r \|f\|_{1,\Omega},
\end{equation}
for each $f\in W^1(\Omega),$ where
$$\|f\|_r =  \left ( \int_\Omega |f|^r \, d\lambda \right )^{1/r}.$$
Now let $\Omega \subseteq \mathbb{C}^n ( \cong \mathbb{R}^{2n} )$ be a smoothly bounded pseudoconvex domain.
 We consider the 
$\ovprt $-complex 
\begin{equation}\label{eq: dbarcomplex1}
L^2(\Omega )\overset{\ovprt }\longrightarrow L^2_{(0,1)}(\Omega)
\overset{\ovprt }\longrightarrow \dots \overset{\ovprt }\longrightarrow
L^2_{(0,n)}(\Omega)\overset{\ovprt }\longrightarrow 0\, ,  
\end{equation}
where $L^2_{(0,q)}(\Omega)$ denotes the space of $(0,q)$-forms on $\Omega$ with
coefficients in $L^2(\Omega).$ The $\ovprt $-operator on $(0,q)$-forms is given by
\begin{equation}\label{eq: deriv1}
\ovprt \left ( \sum_J\,^{'}  a_J \, d\ovli z_J \right )=
\sum_{j=1}^n \sum_J\,^{'}\  \frac{\partial a_J}{\partial \ovli z_j}d\ovli z_j\wedge
d\ovli z_J,
\end{equation}
where $\sum  ^{'} $ means that the sum is only taken over strictly increasing multi-indices $J.$

The derivatives are taken in the sense of distributions, and the domain of $\ovprt $
consists of those $(0,q)$-forms for which the right hand side belongs to
$L^2_{(0,q+1)}(\Omega).$ So $\ovprt $ is a densely defined closed operator\index{closed operator}, and
therefore has an adjoint operator from $L^2_{(0,q+1)}(\Omega)$ into
$L^2_{(0,q)}(\Omega)$ denoted by $\ovprt ^* .$ 

We consider the 
$\ovprt $-complex \index{$\ovprt $-complex}
\begin{equation}\label{eq: complex0}
L^2_{(0,q-1)}(\Omega )\underset{\underset{\ovprt^* }
\longleftarrow}{\overset{\ovprt }
{\longrightarrow}} L^2_{(0,q)}(\Omega )
\underset{\underset{\ovprt^* }
\longleftarrow}{\overset{\ovprt }
{\longrightarrow}} L^2_{(0,q+1)}(\Omega ),
\end{equation}
for $1\le q \le n-1.$

We remark that a $(0,q+1)$-form $u=\sum_{J}^{'} u_J\,d\ovli z_J$ belongs to $\mathcal{C}^\infty_{(0,q+1)}(\ovli \Omega ) \cap {\text{dom}}(\ovprt ^*)$ if and only if
\begin{equation}\label{eq: dom2}
\sum_{k=1}^n u_{kK} \, \frac{\partial r}{\partial z_k} =0
\end{equation}
on $b\Omega$ for all $K$ with $|K|=q,$ where $r$ is a defining function of  $\Omega$ with $|\nabla r(z)|=1$ on the boundary $b\Omega.$ (see for instance \cite{Str})

The complex Laplacian 
$\Box = \ovprt \, \ovprt ^* + \ovprt ^*\,  \ovprt $\index{$\Box$},  defined on the domain
$${\text{dom}}(\Box) = \{ u\in L^2_{(0,q)}(\Omega ) : u\in {\text{dom}}(\ovprt ) \cap  {\text{dom}}(\ovprt^*) ,  \ovprt u\in  {\text{dom}}(\ovprt^*)  ,  \ovprt^* u \in  {\text{dom}}(\ovprt ) \}$$
acts as
an unbounded, densely defined, closed and self-adjoint  operator\index{self-adjoint operator} on 
$L^2_{(0,q)}(\Omega ),$ for $ 1\le q \le n,$ which means that  $\Box = \Box^*$ and ${\text{dom}}(\Box ) =  {\text{dom}}(\Box^*).$  

Note that
\begin{equation}\label{eq: diri1}
(\Box u,u)=( \ovprt \, \ovprt ^* u+ \ovprt ^* \, \ovprt u,u)=\| \ovprt u \|^2 + \| \ovprt ^* u \|^2,
\end{equation}
for $u\in {\text{dom}}(\Box ).$

If $\Omega $ is
  a smoothly bounded pseudoconvex domain in $\mathbb{C}^n,$
  the so-called basic estimate says that
  
 \begin{equation}\label{eq: diri2}
 \| \ovprt u \|^2 + \| \ovprt ^* u \|^2 \ge c \, \|u\|^2,  
 \end{equation}
  for each $u\in {\text{dom}}(\ovprt) \cap {\text{dom}}(\ovprt^* ) , \ c>0.$

 This estimate implies that  $ \Box :  {\text{dom}}(\Box) \longrightarrow L^2_{(0,q)}(\Omega )$\index{$\Box$}
 is bijective and has a bounded inverse 
 $$N:  L^2_{(0,q)}(\Omega ) \longrightarrow {\text{dom}}(\Box). $$
 $N$ is called $\ovprt$-Neumann operator\index{$\ovprt$-Neumann operator}.
In addition 
  \begin{equation}\label{eq: cont5}
 \|N u\| \le \frac{1}{c} \, \|u\|.
 \end{equation}
A different approach to the $\ovprt$-Neumann operator is related to the quadratic form
$$Q(u,v)= (\ovprt u,\ovprt v)+(\ovprt^*u,\ovprt^*v).$$
 For this purpose we consider the embedding
$$j :  {\text{dom}}(\ovprt )\cap  {\text{dom}}(\ovprt^*) \longrightarrow  L^2_{(0,q)}(\Omega ),$$ 
where $ {\text{dom}}(\ovprt )\cap  {\text{dom}}(\ovprt^*)$ is endowed with the graph-norm 
$$u\mapsto (\|\ovprt u\|^2 + \|\ovprt^*u \|^2)^{1/2}.$$ 
The graph-norm stems from the inner product $Q(u,v).$
The basic estimates \eqref{eq: diri2} imply that $j$ is a bounded operator with operator norm 
$$\|j\| \le \frac{1}{\sqrt c}.$$ 
By \eqref{eq: diri2} it follows in addition that $ {\text{dom}}(\ovprt )\cap  {\text{dom}}(\ovprt^*)$  endowed with the graph-norm $u\mapsto (\|\ovprt u\|^2 + \|\ovprt^*u \|^2)^{1/2}$ is a Hilbert space.

The $\ovprt$-Neumann operator $N$ can be written in the form
\begin{equation}\label{eq: dbarN}
N= j \circ j^*,
\end{equation}
details may be found in \cite{Str}.

\section{Compactness and Sobolev inequalities.}~\\

Here we apply a general characterization of compactness of the $\ovprt$-Neumann operator $N$ using a description of precompact subsets in $L^2$-spaces (see \cite{Has5}).

 \begin{theorem}\label{sec: catli} Let $\Omega \subset\subset \mathbb{C}^n$ be a smoothly bounded pseudoconvex domain.  The $\ovprt $-Neumann operator\index{$\dquer$-Neumann operator} $N$ is compact\index{compact operator} if and only if  
   for each $\epsilon >0$ there exists $\omega \subset\subset  \Omega$ such that 

$$ \int_{\Omega \setminus \omega} |u(z)|^2\,d\lambda(z) \le \epsilon ( \|\ovprt u\| ^2 + \| \ovprt^*  u\| ^2)$$

for each $u\in dom\,(\ovprt) \,\cap dom\,(\ovprt^*). $
\end{theorem}

Now let
$$\mathcal{W}^1_{{0,q}}(\Omega ):= \{ u\in L^2_{(0,q)}(\Omega) : u\in dom\,(\ovprt) \,\cap dom\,(\ovprt^*) \}$$ 
endowed with graph norm. As already mentioned above, this "complex" version of a Sobolev space $\mathcal{W}^1_{{0,q}}(\Omega )$ is a Hilbert space.
 
It appears to be interesting to compare the standard Sobolev imbedding
$$W^1(\Omega ) \subset L^r(\Omega) \ , \ r\in [1, 2n/(n-1))$$
where the derivatives are taken with respect of the real variables $x_j =\Re z_j$ and $y_j=\Im z_j$ for $j=1,\dots ,n,$ with the imbedding of the space 
$\mathcal{W}^1_{(0,q)} (\Omega )$ endowed with graph norm, into $L^r_{(0,q)}(\Omega ).$ We have the following result

\begin{theorem}\label{sec: sobo1}
If $\Omega \subset\subset \mathbb{C}^n$ is a smoothly bounded pseudoconvex domain and the inequality
\begin{equation}\label{eq: sobcomp}
\|u\|_r \le C ((\| \ovprt u \|^2 + \| \ovprt ^*u\| ^2)^{1/2}
\end{equation}
for some $r>2$ and for all $u\in dom\,(\ovprt) \,\cap dom\,(\ovprt^*)$ holds, then the $\ovprt$-Neumann operator 
$$N: L^2_{(0,q)}(\Omega) \longrightarrow L^2_{(0,q)}(\Omega)$$
is compact.
\end{theorem}

\begin{proof}
To show this we have to check that the unit ball in $\mathcal{W}^1_{(0,q)} (\Omega )$ is precompact in $L^2_{(0,q)}(\Omega).$ By Proposition  \ref{sec: catli}, we have to show that for each $\epsilon >0$ there exists $\omega \subset\subset \Omega$ such that
$$\int_{\Omega \setminus \omega} |u(z)|^2 \, d\lambda (z) < \epsilon^2,$$
for all $u$ in the unit ball of $\mathcal{W}^1_{(0,q)} (\Omega ).$

By \eqref{eq: sobcomp} and H\"older's inequality\index{H\"older's inequality} we have
\begin{eqnarray*}
\left ( \int_{\Omega \setminus \omega} |u(z)|^2 \, d\lambda (z) \right )^{\frac{1}{2}} &\le & \left (\int_{\Omega \setminus \omega} |u(z)|^r \, d\lambda (z)\right )^{\frac{1}{r}} \cdot |\Omega \setminus \omega|^{\frac{1}{2}-\frac{1}{r}} \\
& \le & C \,  |\Omega \setminus \omega|^{\frac{1}{2}-\frac{1}{r}}.
\end{eqnarray*}
Now we can choose $\omega \subset\subset \Omega$ such that the last term is $< \epsilon.$
\end{proof}
In the following Theorem we suppose that a so-called subelliptic estimate holds. Subelliptic estimates are related to the geometric notion of finite type. We remark that the $\ovprt$-Neumann problem for smoothly bounded strictly pseudoconvex domains is subelliptic with a gain of one derivative for $N$ which is considerably stronger than compactness. 
\begin{theorem}\label{sec: sobol2}
Let $\Omega$ be a  bounded pseudoconvex domain in $\mathbb{C}^n$ with boundary of class $\mathcal{C}^\infty.$ Suppose that $0< \epsilon \le 1/2$ and that
$$ dom\,(\ovprt) \,\cap dom\,(\ovprt^*) \subseteq W^{\epsilon}_{(0,q)} (\Omega ),$$
and that there exists a constant $C>0$ such that
\begin{equation}\label{eq: subell}
\| u \|_{\epsilon, \Omega} \le C (\|\ovprt u \|^2 + \| \ovprt^* u\|^2)^{1/2},
\end{equation}
for all $u \in  dom\,(\ovprt) \,\cap dom\,(\ovprt^*),$ where $W^{\epsilon}_{(0,q)} (\Omega )$ is the standard $\epsilon$-Sobolev space. Then the $\ovprt$-Neumann operator  
$$N: L^2_{(0,q)}(\Omega) \longrightarrow L^2_{(0,q)}(\Omega)$$
is compact and $N$ can be continuously extended as an operator
$$\tilde N: L^{\frac{2n}{n+\epsilon}}_{(0,q)}(\Omega ) \longrightarrow L^{\frac{2n}{n-\epsilon}}_{(0,q)}(\Omega ),$$
which means that there is a constant $C>0$ such that
\begin{equation}\label{eq: lpineq}
\| \tilde N u \|_{\frac{2n}{n-\epsilon}} \le C \, \| u \|_{\frac{2n}{n+\epsilon}},
\end{equation}
for each $u\in  L^{\frac{2n}{n+\epsilon}}_{(0,q)}(\Omega ).$

\end{theorem}

\begin{proof} We use the continuous imbedding  for the space $W^{\epsilon}(\Omega ):$

$$W^{\epsilon}(\Omega ) \longrightarrow L^r(\Omega),$$
for $2\le r \le 2n/(n-\epsilon),$ (see \cite{AF}, Theorem 7.57).
Hence we can choose $r_0 >2$ to get 
$$dom\,(\ovprt) \,\cap dom\,(\ovprt^*) \subseteq W^{\epsilon}_{(0,q)} (\Omega ) \subseteq L^{r_0}_{(0,q)} (\Omega),$$
and we can apply Theorem \ref{sec: sobo1}.
\vskip 0.5 cm

To show that $N$ extends continuously recall that $N=j \circ j^*,$ where 
$$j: dom\,(\ovprt) \,\cap dom\,(\ovprt^*) \longrightarrow L^2_{(0,q)}(\Omega),$$
see \cite{Str}. In our case $j$ is a continuous operator into  $L^{\frac{2n}{n-\epsilon}}_{(0,q)}(\Omega ),$ hence
$$j^* :  L^{\frac{2n}{n+\epsilon}}_{(0,q)}(\Omega ) \longrightarrow dom\,(\ovprt) \,\cap dom\,(\ovprt^*),$$
which proves the assertion. 
\end{proof}

St. Krantz \cite{KR2}, R. Beals, P.C. Greiner and N.K. Stanton \cite{BGS}, I.Lieb and R.M. Range \cite{LR}, and A. Bonami and N. Sibony \cite{BoSi} proved $L^p$-estimates and Lipschitz estimates for solution operators of the inhomogeneous $\ovprt$-equation and the $\ovprt$-Neumann operator using integral representations for the kernel of these operators, but without relationship to compactness and continuous extendability.

\begin{rem} If $\Omega$ is a bounded strictly pseudoconvex domain in $\mathbb{C}^n$ with boundary of class $\mathcal{C}^\infty,$ then \eqref{eq: subell} is satisfied for $\epsilon = 1/2$ (see \cite{Str}, Proposition 3.1).

D'Angelo (\cite{DA79}, \cite{DA82}) and Catlin \cite{Ca83}, \cite{Ca84}, \cite{Ca87}) give a characterization of when a subelliptic estimate holds in terms of the geometric notion of finite type, see also \cite{Str}.

\end{rem}

\begin{cor}\label{sec: nonco}
Let $\Omega$ be a smooth bounded pseudoconvex domain in $\mathbb{C}^n, \ n \ge 2.$ Let $P \in b\Omega$ and assume that there is an m-dimensional complex manifold $M\subset b\Omega$ through $P \ (m\ge 1),$ and $b\Omega$ is strictly pseudoconvex at $P$ in the directions transverse to $M$ (this condition is void when $n=2$). Then \eqref{eq: sobcomp} is not satisfied for $(0,q)$-forms with $1\le q \le m.$ 

\end{cor}
\begin{proof} Theorem 4.21 of \cite{Str} gives that the $\ovprt$-Neumann operator fails to be compact on $(0,q)$-forms with $1\le q \le m.$ Hence we can again apply Proposition \ref{sec: sobo1} to get the desired result.
\end{proof}

\begin{rem} If the Levi form of the defining function of $\Omega$ is known to have at most one degenerate eigenvalue at each point (the eigenvalue zero has multiplicity at most 1), a disk in the boundary is an obstruction to compactness of $N$ for $(0,1)$-forms. A special case of this is implicit in \cite{Kim} for domains fibered over a Reinhardt domain in $\mathbb{C}^2.$
\end{rem}

\section*{Acknowledgement}

 The author wishes to express his gratitude to the referee for helpful suggestions.

\bibliographystyle{amsplain}
\bibliography{mybibliography}

\end{document}